\documentclass[11pt]{article}

\usepackage{amsmath}

\usepackage[english]{babel}
\usepackage[ansinew]{inputenc}
\usepackage{epsfig}
\usepackage{color,graphicx}
\usepackage{wrapfig}
\usepackage{amssymb,amsthm,amsmath}
\usepackage{dsfont}
\usepackage{natbib}
\usepackage{abstract}
\usepackage{enumerate}
\usepackage{comment}

\newtheorem{teo}{Theorem}
\newtheorem{cor}{Corollary}
\newtheorem{lema}{Lemma}
\newtheorem{prop}{Proposition}

\newtheorem{definition}{Definition}
\newtheorem*{definition*}{Definition}
\newtheorem*{thm*}{Theorem}
\newtheorem*{proposition*}{Proposition}

\newcommand{\R}{\mathbb{R}}

\newcommand{\N}{\mathbb{N}}

\def\ds{\displaystyle}

\begin{document}

\title{On the definition of likelihood function}

\author{\bf{F. B. Gon\c{c}alves$^{a}$, P. Franklin$^{b}$}}
\date{}

\maketitle

\begin{center}
{\footnotesize $^a$ Universidade Federal de Minas Gerais, Brazil\\
$^b$ Universidade Federal de Uberl\^{a}ndia}
\end{center}

\begin{abstract}

We discuss a general definition of likelihood function in terms of Radon-Nikod\'{y}m derivatives.
The definition is validated by the Likelihood Principle once we establish a result regarding the proportionality of likelihood functions under different dominating measures.
This general framework is particularly useful when there exists no or more than one obvious choice for a dominating measure as in some infinite-dimensional models.
We discuss the importance of considering continuous versions of densities and how these are related to the Likelihood Principle and the basic concept of likelihood. We also discuss the use of the predictive measure as a dominating measure in the Bayesian approach.
Finally, some examples illustrate the general definition of likelihood function and the importance of choosing particular dominating measures in some cases.


\end{abstract}

\section{Introduction}

In this paper, we shall discuss some mathematical foundations of Likelihood Theory, more specifically, the definition of likelihood function. Likelihood-based methodologies are undoubtedly the most common and often most efficient ones to perform statistical inference - in particular, maximum-likelihood estimation and Bayesian inference. This is due to general strong properties of the likelihood function that stem from a solid mathematical foundation, based on measure/probability theory.

The concept of likelihood goes back to Fisher with the actual term first appearance in \cite{Fisher21}, so before Kolmogorov's probability axioms \citep{komol33} and the Radon-Nikod\'{y}m Theorem \citep{RNT}. Nevertheless, the intuition given by Fisher to construct the concept of likelihood made it straightforward to extend the definition of likelihood function (LF) in terms of Radon-Nikod\'{y}m derivatives. The earliest explicit version of such definition we could find is from \cite{Lind53} [Definition 2.4], however, it is implicitly assumed for example in \cite{HS49}. It consists in defining the likelihood function as any Radon-Nikod\'{y}m (RN) derivative (see Definition \ref{LFd} in Section \ref{secLPT}), i.e. using any $\sigma$-finite dominating measure.

Since any model that has a dominating measure admits an uncountable number of dominating measures, the aforementioned definition of likelihood function can only be admissible if the choice of dominating measure has no influence in the inference process. Under the Likelihood Principle (LP), it means that any two distinct dominating measures should lead to proportional likelihood functions. Although such a result is accepted by the statistical community, it has not yet been properly stated, proven or explored. This is one of the specific aims of this paper. In fact, this issue has never been properly raised in the literature. The general definition of likelihood is always approached by assuming the existence of a common dominating measure and there is no mention of other measures or what would be the implications of a making a different choice. \citet{Reid} mentions that ``Some books describe the likelihood function as the Radon-Nikod\'{y}m derivative of the probability measure with respect to a dominating measure. Sometimes the dominating measure is taken to be $P_{\theta_0}$ for a fixed value $\theta_0\in\Theta$. When we consider probability spaces and/or parameter spaces that are infinite dimensional, it is not obvious what to use as a dominating measure."

We state and prove what we call the Likelihood Proportionality Theorem, which validates (in terms of the LP) the general definition of likelihood function in terms of Radon-Nikod\'{y}m derivatives. Moreover, we discuss how continuous RN derivatives are relevant when obtaining the likelihood function. More specifically, we present some results showing that the continuity property guarantees the proportionality result and leads to likelihood functions that carry the intuitive concept of likelihood. We also discuss the use of the prior predictive measure as a dominating measure in a Bayesian context.

Finally, we discuss and provide several examples where the choice of the dominating measure requires special attention. In particular, situations that require some effort to find a valid dominating measure to obtain a valid likelihood function and situations in which more than one obvious dominating measure is available but a particular choice may significantly easy the inference process. We also emphasise that we work with Likelihood Theory in a general context and not just for parametric models. This context is considered in several relevant inference problems nowadays (specially infinite-dimensional problems under the Bayesian approach), as we illustrate in some of the examples provided.

We discuss four general classes of widely used models. The first example considers general finite-dimensional models and describes how to obtain a valid likelihood function when dealing with point-mass mixtures. The second example discusses some results regarding dominating measures for the exponential family. The third example explores possibly important implications of the choice of the dominating measure in general missing data problems. Finally, the last two examples consider classes of infinite-dimensional models: Poisson processes and diffusion processes.

Other works in the context of mathematical aspects of the likelihood function but that pursue different directions can be found in \citet{Nielsen}, \citet{Naderi96}, \citet{Naderi97} and \citet{Naderi07}.

This paper is organised as follows: Section 2 presents the Likelihood Proportionality Theorem and some important auxiliary results; Section 3 discusses the importance of continuous RN derivatives and Section 4 discusses the use of the predictive measure as a choice for dominating measure; Section 5 presents some examples regarding the choice of dominating measure and Section 6 brings final remarks.

\section{The Likelihood Proportionality Theorem}\label{secLPT}

Let $(\Omega, \mathcal{F})$ denote a measurable space, $(\Omega,\mathcal{F},\mu)$ a measure space and $M(\Omega,\mathcal{F})$ the collection of all measurable functions $f:\Omega\longrightarrow\R$.

\begin{definition*}[Statistical model]
A statistical model is a family of probability measures $\mathcal{P}$ on $(\Omega,\mathcal{F})$, i.e $\mathcal{P}=\{P_{\theta};\ \theta\in\Theta\}$, where the $P_{\theta}$'s are probability measures and $\Theta$ is an arbitrary index set. In the particular case where $\Theta\subset\mathds{R}^d$ for $d\in\mathds{N}$, $\mathcal{P}$ is called a parametric model, $\theta$ a parameter and $\Theta$ the parameter space. In any other case $\mathcal{P}$ is called a non-parametric model.
\end{definition*}

An statistical model is called identifiable if mapping from $\Theta$ to $\mathcal{P}$ is a bijection. This is a highly desirable property to perform statistical inference. In particular, it is one of the regularity conditions required in the most important results regarding maximum likelihood estimation.

A statistical inference problem can be generally described as follows. Given a model $\mathcal{P}$, one wants to estimate a probability measure $P_{\theta^*}\in\mathcal{P}$ (often referred to as population) based on a sample (realisation(s) from $P_{\theta^*}$ - a random experiment). The likelihood function (of $P_{\theta}$) is one way to quantify the likelihood of each $P_{\theta}$ having generated the data. We formally define the likelihood function as follows.

\begin{definition}[Likelihood function]\label{LFd}
Let $\mathcal{P}=\{P_{\theta};\ \theta\in\Theta\}$ be a statistical model and $\nu$ any $\sigma$-finite measure such that $\mathcal{P}<<\nu$. For a given observed sample point $\omega$, the likelihood function $l_{\nu}(\theta;\omega)$ for $P_{\theta}\in\mathcal{P}$ is given by a version of the Radon-Nikod\'{y}m derivative $\ds \frac{dP_{\theta}}{d\nu}(\omega)$, for all $\theta\in\Theta$.
\end{definition}

As we have mentioned before, it is reasonable to extend the intuition developed by Fisher to construct the concept of likelihood to the definition above. However, a formal validation of this definition is to be achieved through the Likelihood Principle and the Likelihood Proportionality Theorem. The LP specifies how the likelihood function ought to be used for data reduction - a detailed addressing of the LP can be found in \cite{BE&WO}.\\

\noindent\textbf{The Likelihood Principle.} \cite{BE&WO} [page 19] \emph{``All the information about $P_{\theta}$ obtainable from an experiment is contained in the likelihood function for $P_{\theta}$ given the sample. Two likelihood functions for $P_{\theta}$ contain the same information about $P_{\theta}$ if they are proportional to one another."}\\

The proportionality mentioned in the LP means that $\ds l_1(\theta;\omega)=h(\omega)l_2(\theta;\omega)$, with $l_1$ and $l_2$ being the two likelihood functions. Furthermore, the fact that both likelihood functions contain the same information about $P_{\theta}$ imply that the same inference must be done. The version of the LP stated above is a general version and may be contextualised in different cases, leading to more specific versions. For example, the two likelihood functions may refer to two different data points $\omega_1$ and $\omega_2$, such that $l(\theta;\omega_1)\propto l(\theta,\omega_2)$, or even different experiments which, in our construction, could be characterised as the sample consisting of observing different functions $f\in M(\Omega,\mathcal{F})$, but also leading to proportional likelihood functions. In this work, however, we consider the LP under the perspective of different dominating measures used to obtain the likelihood function, leading to the following version.\\

\noindent\textbf{The Likelihood Principle} (for different dominating measures). \emph{Let $\nu_1$ and $\nu_2$ be two dominating measures for $P_{\theta}$, for all $\theta\in\Theta$, leading to likelihood functions $l_{\nu_1}$ and $l_{\nu_2}$, such that $l_{\nu_1}(\theta;\omega)\propto l_{\nu_2}(\theta;\omega)$, for all $\theta\in\Theta$ and for all $\omega\in\Omega_0$ such that $P_{\theta}(\Omega_0)=1$ for all $\theta\in\Theta$. Then, $l_{\nu_1}$ and $l_{\nu_2}$ contain the same information about $P_{\theta}$.}\\

This way, Definition \ref{LFd} is validated by the LP if different dominating measures lead to proportional likelihood functions with probability 1 under all $P_{\theta}$. Such a result is stated in detail in the Likelihood Proportionality Theorem further ahead in this section.

Before stating and proving the theorem, we need some auxiliary results. The first one is a neat result from \cite{HS49} (Lemma 7) considering dominated families of measures.

\begin{lema}\label{lema1}
(Halmos and Savage, 1949) Let $\mathcal{P}=\{P_{\theta};\ \theta\in\Theta\}$ be a family of probability measures and $\nu$ a $\sigma$-finite measure on $(\Omega,\mathcal{F})$. If $\mathcal{P}<<\nu$ then there exists a probability measure $Q$, such that $\mathcal{P}<<Q$ and $Q=\sum_{i=1}^{\infty} c_{i}P_{\theta_i}$, where the $c_{i}$'s are non-negative constants with $\sum_{i=1}^{\infty} c_{i}=1$ and $P_{\theta_i}\in \mathcal{P}$.
\end{lema}
\begin{proof}
\citep[See][page 53]{Jorgensen12}.
\end{proof}

Lemma \ref{lema1} provides a strong property for families of probability measures that are dominated by a common $\sigma$-finite measure. The result establishes the existence of a countable coverage for that family. The key feature in that property is the fact that $P_{\theta_i}$-almost surely, for $i=1,2,\ldots$, implies in $P_{\theta}$-almost surely, for all $\theta\in\Theta$, and this is crucial to establish the main result in this paper - the Likelihood Proportionality Theorem.

\begin{definition}\label{defmdm}
For a family of probability measures $\mathcal{P}=\{P_\theta;\ \theta\in\Theta\}$, suppose that the family $\Upsilon=\{\nu;\ \mathcal{P}<<\nu\}$ is non-empty. If there exists $\lambda\in\Upsilon$ such that $\lambda<<\nu$ for all $\nu\in\Upsilon$, then we say that $\lambda$ is a minimal dominating measure for the family $\mathcal{P}$.
\end{definition}
Note that a minimal dominating measure is not necessarily unique. However, by definition, two minimal dominating measures are always equivalent.

\begin{prop}\label{propmdm}
Let $\mathcal{P}=\{P_\theta;\ \theta\in\Theta\}$ be a family of probability measures defined on the measurable space $(\Omega,\mathcal{F})$. Suppose that the family $\Upsilon=\{\nu;\ \mathcal{P}<<\nu\}$ is non-empty. Then, there exists a minimal dominating measure $\lambda$ for $\mathcal{P}$.
\end{prop}

\begin{proof}
See Appendix.
\end{proof}

For a function $f$ in $M(\Omega,\mathcal{F})$, define $[f]_\mu$ as the equivalence class of $f$ with respect to $\mu$, i.e. the collection of all functions $g$ in $M(\Omega,\mathcal{F})$ such that $g=f$ $\mu$-a.s. We now state and prove the Likelihood Proportionality Theorem.

\begin{teo}[The Likelihood Proportionality Theorem]\label{main.theo}
Let $\mathcal{P}=\{P_{\theta};\ \theta\in\Theta\}$ be a family of probability measures and $\nu_1,\nu_2$ $\sigma$-finite measures on $(\Omega,\mathcal{F})$. Suppose that $\mathcal{P}<<\nu_1$, $\mathcal{P}<<\nu_2$ and that $\nu$ is a minimal dominating measure for $\mathcal{P}$. Then, there exists a measurable set $A$ such that $P_\theta(A)=1$, for all $\theta\in\Theta$, and there exist $f_{1,\theta}\in [\frac{dP_{\theta}}{d\nu_1}]_{\nu}$, $f_{2,\theta}\in [\frac{dP_{\theta}}{d\nu_2}]_{\nu}$, for all $\theta\in\Theta$, and a measurable function $h$ such that
\begin{equation}\label{main.eq}
f_{1,\theta}(\omega)=h(\omega)f_{2,\theta}(\omega),\ \forall\theta\in\Theta,\ \forall\omega\in A.
\end{equation}
\end{teo}
\begin{proof}
See Appendix.
\end{proof}

\noindent\textbf{Discussion of Theorem \ref{main.theo}.} Note that equation (\ref{main.eq}) implies that $f_{1,\theta}(\omega)\propto_\theta f_{2,\theta}(\omega)$, $\forall\theta\in\Theta,\;\forall\omega\in A$, which validates Definition \ref{LFd} in terms of the Likelihood Principle i.e., independent of the choice of the dominating measure the inference will (a.s.) be the same. Furthermore, the proportionality result is valid a.s. $P_{\theta}$, for all $\theta\in\Theta$, in particular, for the true $\theta$ (whichever it is).

Note, however, that Theorem \ref{main.theo} states the existence of versions of RN derivatives that satisfies (\ref{main.eq}), which means that not all versions necessarily do. In this sense, it would be useful to define a class of versions that always satisfies (\ref{main.eq}) and, possibly, lead to a well-behaved likelihood function, for example, that satisfies the classical regularity conditions (if such a version exists). We further explore this issue in Section \ref{seccontv}, considering continuous versions of RN derivatives.

In some cases, $[\frac{dP_\theta}{d\nu_1}]_{\nu}$ and $[\frac{dP_\theta}{d\nu_2}]_{\nu}$ are unitary sets. For example, in a family $\mathcal{P}=\{P_\theta;\ \theta\in\Theta\}$ of discrete distributions, i.e. $P_\theta(\omega)>0$, for all $\theta\in\Theta$ and for all $\omega\in\Omega$. In those cases, the unique version of the respective RN derivative is $[\frac{dP_\theta}{d\nu_i}]_{\nu}=\frac{P_{\theta}(\omega)}{\nu_i(\omega)}$, for all $\omega\in A$, with $A=\{\omega\in\Omega: \nu(\omega)>0\}$.
Another interesting particular example is the case where the family of probability measures is a countable set. In this case, any pair of versions of the RN derivative (one for each dominating measure) satisfies (\ref{main.eq}).

We also call the reader's attention to an important issue raised by one of the referees of this paper. Note that, in our context, the two likelihood functions in Theorem \ref{main.theo} are obtained from a common statistical model, therefore, with common sample space. This means that the version of the LP in terms of different dominating measures that, along with Theorem \ref{main.theo}, validates the definition of likelihood function in Definition \ref{LFd}, is a weak version of the LP. As a consequence, for example, not only the maximum likelihood estimator (MLE) will be the same under both dominating measures, but also its distribution, and therefore, all the inference based on this distribution. This also applies to any other frequentist estimator whose definition is independent of the dominating measure.

We can relate Theorem \ref{main.theo} to the Factorisation Theorem by stating the following Proposition.

\begin{prop}\label{T1FT}
Consider $\mathcal{P}$, $\nu_1$, $\nu_2$ and $\nu$ from Theorem \ref{main.theo} and $Q$ from Lemma \ref{lema1} and let $T$ be a sufficient statistic for $\mathcal{P}$ with range space $(\mathcal{T},\mathcal{B})$. Then:
\begin{enumerate}[i)]
  \item For each version $g_{\theta}^*\in[\frac{dP_{\theta}}{dQ}]_{\nu}$ in $(\Omega,\sigma(T))$ and $h_1\in[\frac{dQ}{d\nu_1}]_{\nu}$ in $(\Omega,\mathcal{F})$, there exists a $\mathcal{B}$-measurable function $g_{\theta}$ such that $g_{\theta}^*=g_{\theta}\circ T$ and the function $f_{1,\theta}=(g_{\theta}\circ T)h_1$ is a version in $[\frac{dP_{\theta}}{d\nu_1}]_{\nu}$, for all $\theta$.
  \item If we obtain $f_{1,\theta}$ and $f_{2,\theta}$ as in $i)$ (for $\nu_1$ and $\nu_2$, respectively) from the same $g_{\theta}^*$, then $f_{1,\theta}\propto f_{2,\theta}$ in a measurable set $A$, for all $\theta\in\Theta$, such that $\nu(A^c)=0$.
\end{enumerate}
\end{prop}
\begin{proof}
See Appendix.
\end{proof}

Part $i)$ from Proposition \ref{T1FT} can be seen as a stronger version of the Factorisation Theorem as it states that the density representation is valid for all $\theta$ in the whole $\Omega$, i.e. it holds $P_{\theta}$ a.s., for all $\theta$. The classical version of the Factorisation Theorem is a consequence, since all versions in $[\frac{dP_{\theta}}{d\nu_1}]_{\nu_1}$ are $P_{\theta}$ equivalent.

Finally, note that the result in Theorem \ref{main.theo} is valid for any topological structure induced in the sample space $\Omega$, in particular, if $\Omega$ is non-separable and/or non-metric.

\section{Continuous versions of Radon-Nikod\'{y}m derivatives}\label{seccontv}

As we have mentioned before, we would like to define a subclass of RN versions that would always satisfy the proportionality relation (\ref{main.eq}) and, therefore, provide a practical way to obtain a likelihood function. That is achieved by considering continuous versions of densities. We state two results (Theorem \ref{CVT} and Proposition \ref{TCVLF2}) that, under different assumptions, guarantee that continuous versions of the RN derivatives, when these exist, do satisfy (\ref{main.eq}). In fact, in \cite{Picc82} and \cite{Picc83}, the likelihood function is defined as a continuous version of the RN derivative. The author proves that, if such a version exists, it is unique (under some additional assumptions) and this particular definition is justified by the fact that such a version is related to a limit that builds on the intuition of likelihood. \citet{BE&WO} suggest the use of continuous versions in face of the ambiguity implied by the existence of different versions of RN derivatives. Their choice is justified as follows: ``By restricting our attention to ($\nu$-almost everywhere) continuous densities, continuous sufficient statistics, etc. we could develop versions of the conditionality, sufficiency, and likelihood principles very similar to those in the discrete setting." Finally, regarding well-behaved versions of the likelihood function, continuity (of the likelihood) is a particular property of interest. In particular, most of the important results regarding properties of the MLE rely on assumptions that include continuity. In some cases (especially for parametric models), continuity (in $\theta$) of the likelihood is implied by continuity (in $\omega$) of the RN density.

For the whole of this section, let $\Omega$ be a metric separable space with a distance that induces the topology $\textbf{A}$. As usual, $\mathcal{F}$ is the smallest $\sigma$-algebra containing $\textbf{A}$ - the Borel $\sigma$-algebra of $\Omega$.

We now discuss why continuous versions of densities lead to likelihood functions that carry the true intuition of likelihood. In the simplest case where $\Omega$ is discrete, the likelihood is proportional to the probability of the observed sample $\omega_0$, which gives a clear interpretation to the concept of likelihood. This concept is extended to the continuous case by considering the following limit:
\begin{equation}\label{ExLim}
\lim_{n\rightarrow\infty}\frac{P_{\theta}(A_n)}{\nu(A_n)},
\end{equation}
for all sequence $A_1,A_2,\ldots$ such that $A_n\in\mathfrak{A}(\omega_0)$ and $A_1\subset A_2\subset \ldots$, where $\mathfrak{A}$ is the collection of open neighbourhoods of $\omega_0$. \cite{Picc82} shows that there exists a continuous version $f_{\theta}^c$ of $dP_{\theta}/d\nu$ if and only if there exists the limit in (\ref{ExLim}), in which case $f_{\theta}^c(\omega_0)$ is exactly this limit.

It is natural to expect that continuous versions will satisfy the proportionality relation (\ref{main.eq}). This is established in Theorem \ref{CVT} and Proposition \ref{TCVLF2} below. In order to prove these two results, we require the following Lemma and definitions (which are valid for general sample spaces $\Omega$).

\begin{definition}
Let $(\Omega,\mathcal{F},\mu)$ be a measure space and $A\in\mathcal{F}$ a nonempty set. We denote $\mu\big|_A$ as the restriction of the measure $\mu$ on $(A,\mathcal{F}(A))$, i.e., $\mu\big|_A$ is the measure defined on $(A,\mathcal{F}(A))$ such that $\mu\big|_A(B)=\mu(B),\ \forall B\in\mathcal{F}(A)$.
\end{definition}

\begin{lema}\label{theo.2}
Let $\mathcal{P}=\{P_{\theta};\ \theta\in\Theta\}$ be a family of probability measures and $\nu_1$ and $\nu_2$ $\sigma$-finite measures on $(\Omega,\mathcal{F})$, where $\Theta$ is a nonempty set. Suppose that $\mathcal{P}<<\nu_1$ and $\mathcal{P}<<\nu_2$. Then, there exists a measurable set $A$ such that
\begin{enumerate}[(i)]
    \item $P_\theta(A)=1$, for all $\theta\in\Theta$ and
    \item $\nu_1\big|_A$ and $\nu_2\big|_A$ are equivalent measures, that is, $\nu_1\big|_A<<\nu_2\big|_A$ and $\nu_2\big|_A<<\nu_1\big|_A$.
\end{enumerate}
\end{lema}
\begin{proof}
See Appendix.
\end{proof}

\begin{definition}[Dominating pair]\label{FS}
Consider $\mathcal{P}=\{P_{\theta};\ \theta\in\Theta\}$, where $\Theta$ is a nonempty set, to be a family of probability measures and let $\nu_1$ and $\nu_2$ be $\sigma$-finite measures on $(\Omega,\mathcal{F})$ such that $\mathcal{P}<<\nu_1$ and $\mathcal{P}<<\nu_2$. A pair $(A,\nu)$ is called a dominating pair for the triple $(\mathcal{P},\nu_1,\nu_2)$, where $A\in\mathcal{F}$ and $\nu=\sum_{i=1}^{\infty} c_{i}P_{\theta_i}$ is a minimal dominating measure, for some sequences $\{\theta_i\}_{i=1}^{\infty}$ and $\{c_i\}_{i=1}^{\infty}$ such that $\sum_{i=1}^{\infty} c_{i}=1$, if $\nu_1\big|_A$, $\nu_2\big|_A$ and $\nu\big|_A$ are equivalent and $\nu(A)=1$.
\end{definition}

Note that a dominating pair for $(\mathcal{P},\nu_1,\nu_2)$ always exists. That is guaranteed by Proposition \ref{propmdm} and its proof and Lemma \ref{theo.2}. The use of a dominating pair is crucial to establish the proportionality of likelihood functions obtained from continuous versions of RN derivatives, as stated in the following Theorem.

\begin{teo}\label{CVT}
Let $(A,\nu)$ be a dominating pair for $(\mathcal{P},\nu_1,\nu_2)$. If there exist continuous versions of Radon-Nikod\'{y}m derivatives $f_{1,\theta}\in [\frac{dP_{\theta}|_A}{d\nu_1|_A}]_{\nu|_A}$, $f_{2,\theta}\in [\frac{dP_{\theta}|_A}{d\nu_2|_A}]_{\nu|_A}$, $\forall\theta\in\Theta$, then, for all $h\in[\frac{d\nu_2|_A}{d\nu_1|_A}]_{\nu|_A}$, there exists a measurable set $B_h\in\mathcal{F}(A)$ such that $P_\theta(B_h)=1$, for all $\theta\in\Theta$, $h$ is continuous on $B_h$ and

$$f_{1,\theta}(\omega)=h(\omega)f_{2,\theta}(\omega),\  \forall\theta\in\Theta,\ \forall\omega\in B_h.$$
\end{teo}

\begin{proof}
See Appendix.
\end{proof}

Theorem \ref{CVT} defines a specific subclass of RN versions (the one with the continuous versions) that always satisfies the proportionality relation (\ref{main.eq}). Moreover, if the dominating measures under consideration are locally finite (LF) - see Appendix A, the continuous version (w.r.t. each of the measures) is unique (guaranteed by Theorem \ref{uniqcv} - see Appendix A). In many statistical models, there exist, and it is straightforward to obtain, continuous versions of $f_{1,\theta}$ and $f_{2,\theta}$ in $\Omega$, for all $\theta\in \Theta$.

Let $S_\nu$ be the support of a measure $\nu$ on $(\Omega,\mathcal{F})$ (see Appendix A for the formal definition of support and related results).
The following corollary applies to several examples of statistical models.
\begin{cor}\label{CVTcor1}
Suppose that $\nu_1$ and $\nu_2$ are LF measures with $S_{\nu_1}=\Omega$. Suppose also that there exist continuous versions of Radon-Nikod\'{y}m derivatives $f_{1,\theta}\in [\frac{dP_{\theta}}{d\nu_1}]_{\nu}$, $f_{2,\theta}\in [\frac{dP_{\theta}}{d\nu_2}]_{\nu}$, for all $\theta\in\Theta$, where $\nu$ is a minimal dominating measure, and that $f_{1,\theta}(\omega)>0$ and $f_{2,\theta}(\omega)>0$, for all $\omega\in\Omega$ and $\theta\in\Theta$. Then

$$f_{1,\theta}(\omega)\propto_{\theta}f_{2,\theta}(\omega),\;\;\forall\omega\in\Omega,\;\forall\theta\in\Theta.$$
\end{cor}

\begin{proof}
See Appendix.
\end{proof}

\begin{prop}\label{TCVLF2}
Let $\mathcal{P}=\{P_{\theta};\ \theta\in\Theta\}$ be a family of probability measures and $\nu_1$ and $\nu_2$ LF measures on $(\Omega,\mathcal{F})$, where $\Theta$ is a nonempty set, $\mathcal{P}<<\nu_1$, $\mathcal{P}<<\nu_2$. Let $(\nu_3,A)$ be a dominating pair for $(\mathcal{P},\nu_1,\nu_2)$ and $S_{\theta}$, $S_{1}$, $S_{2}$ and $S_{3}$ be the supports of $P_{\theta}$ (for each $\theta\in\Theta$), $\nu_1$, $\nu_2$ and $\nu_3$, respectively. If there exists a continuous version on $S_{\theta}$ of the Radon-Nikod\'{y}m derivative $f_{2,\theta}\in [\frac{dP_{\theta}|_{S_{\theta}}}{d\nu_2|_{S_{\theta}}}]_{\nu_1|_{S_{\theta}}}$, $\forall\theta\in\Theta$, and there exists a continuous version on $S_{3}$ of the Radon-Nikod\'{y}m derivative $h\in[\frac{d\nu_2|_{S_{3}}}{d\nu_1|_{S_{3}}}]_{\nu_1|_{S_{3}}}$, then $f_{2,\theta}$ and $h$ are unique in $S_{\theta}$ and $S_3$, respectively, and there exists an unique continuous version of $f_{1,\theta}\in [\frac{dP_{\theta}|_{S_{\theta}}}{d\nu_1|_{S_{\theta}}}]_{\nu_1|_{S_{\theta}}}$ on $S_{\theta}$, for all $\theta\in\Theta$. Moreover, defining
$\Phi_\omega=\{\theta\in\Theta;\ \omega\in S_{\theta}\}$, we have that $f_{1,\theta}(\omega)$ and $f_{2,\theta}(\omega)$ are proportional for every $\theta\in\Phi_\omega$.
\end{prop}
\begin{proof}
Simply note that $S_{\theta} \subset S_{3}$ (Proposition \ref{TCVLF} - see Appendix A) and define $f_{1,\theta}(\omega)=h(\omega)f_{2,\theta}(\omega)$, $\forall\omega\in S_{\theta},\;\forall\theta\in \Theta$. The uniqueness of $f_{1,\theta}$, $f_{2,\theta}$ and $h$ is guaranteed by Theorem \ref{uniqcv} (see Appendix A).
\end{proof}

\section{The predictive measure as a dominating measure}

\citet{Rafael} propose a novel methodology for non-parametric density ratio estimation and show how this general framework
can be extended to address the problem of estimating the likelihood function when this is intractable.
In particular, the authors use the density of the prior predictive measure in the denominator of
the ratio and, therefore, obtain an approximation for the likelihood function induced by the use of this particular dominating measure. We now investigate when the prior predictive measure can be used as a dominating measure for the model.

Let $X$ be a sample from a population in a parametric family $\mathcal{P}=\{P_\theta;\ \theta\in\Theta\}$, where $\Theta\subset {\R}^k$ for a fixed $k\in {\N}$ and $\mathcal{X}$ be the range of $X$. Let $R$ be a non-zero prior distribution on $\Theta$ and denote by $\mathcal{B}_{\mathcal{X}}$ and $\mathcal{B_{\theta}}$ the $\sigma$-fields on $\mathcal{X}$ and $\Theta$, respectively. Suppose that the function $P_\theta(B):\Theta\longmapsto[0,1]$ is Borel for any fixed $B\in\mathcal{B}_\mathcal{X}$. Then, there is a unique probability measure $P$ on $(\mathcal{X}\times\Theta,\sigma(\mathcal{B}_\mathcal{X}\times\mathcal{B}_\Theta))$ (\cite{Shao}, Chapter 4) such that, for $B\in\mathcal{B}_\mathcal{X}$ and $C\in\mathcal{B}_\Theta$, $P(B\times C)=\int_C P_\theta(B)dR$.
The posterior distribution of $\theta$ given $X=x$ - denoted by $P_{\theta|x}$, is obtained by the Bayes Formula.

\emph{(Bayes Formula) \; Assume that $\mathcal{P}$ is dominated by a $\sigma$-finite measure $\nu$ and $f_\theta(x)=\frac{dP_\theta}{d\nu}(x)$ is a Borel function on $(\mathcal{X}\times\Theta,\sigma(\mathcal{B}_\mathcal{X}\times\mathcal{B}_\Theta))$. Suppose that $m(x)=\int_\Theta f_\theta(x)dR>0$. Then, the posterior distribution $P_{\theta|x}$ is dominated by $R$ and
$$\frac{dP_{\theta|x}}{dR}(\theta)=\frac{f_\theta(x)}{m(x)}.$$}

The function $m$ in the Bayes Formula (BF) is called the marginal p.d.f. of $X$ with respect to $\nu$. Note that the p.d.f in the BF is well-defined only for the points $X=x$ such that $m(x)>0$. In fact, for a value $x$ such that $m(x)=0$, the likelihood function vanishes $R$-almost everywhere. Simply note that, if $m(x)=0$ then $\int_\Theta f_\theta(x)dR=0$ and, since $R(\Theta)>0$, we have that $R(Z_{x}^c)=0$, for $Z_x=\{\theta\in\Theta;\ f_\theta(x)=0\}$.

The zero set of the function $m$ actually plays an important role for the predictive measure $\lambda$, which is defined on $(\mathcal{X},\mathcal{B}_\mathcal{X})$ by $\lambda(A)=\int_A m d\nu,\ \forall A\in\mathcal{B}_\mathcal{X}$. The following four results relate the predictive measure to the context of dominating measures.
\begin{prop}\label{PMDM1}
The predictive measure is independent of the choice of the measure that dominates the population $\mathcal{P}$.
\end{prop}
\begin{proof}
See Appendix.
\end{proof}

\begin{prop}\label{prop.predictive}
If $m(x)>0$ for all $x\in\mathcal{X}$, then the predictive measure $\lambda$ dominates $\mathcal{P}$.
\end{prop}
\begin{proof}
Follows directly from the definition of $m$.
\end{proof}

Note that, for $\lambda$ to dominate $\mathcal{P}$, it is enough to have $\nu(N)=0$. Nevertheless, the result is not guaranteed if we only have that $P_\theta(N)=0$ $R$-almost everywhere.

\begin{teo}\label{PMDM2}
The predictive measure $\lambda$ dominates $P_\theta$ if and only if $P_\theta(N)=0$ and, therefore, $\lambda$ dominates $\mathcal{P}$ if and only if $P_\theta(N)=0$ for all $\theta\in\Theta$.
\end{teo}
\begin{proof}
See Appendix.
\end{proof}

The following result is of more practical use.

\begin{teo}\label{PMDM3}
If $M_\theta=\{x\in\mathcal{X};\ f_\theta(x)>0\}$ does not depend on $\theta$, then $\mathcal{P}<<\lambda$.
\end{teo}
\begin{proof}
See Appendix.
\end{proof}

\section{Examples}\label{secexamp}

We now explore the results presented in this paper through examples. We consider models for which the choice of the dominating measure require some effort and/or instigate some interesting discussion. The Likelihood Proportionality Theorem is implicitly applied to guarantee that valid likelihood functions are obtained and the continuity of the densities discussed in Section \ref{seccontv} is highlighted.

\subsection{Finite-dimensional random variables}

It is often the case in which the statistical model under consideration is a family of probability measures consisting of a finite-dimensional random variable with discrete and/or continuous coordinates. This covers a wide range of models from iid univariate random variables to highly structured hierarchical Bayesian models with mixture components. In this case, the most common choice for dominating measure is the appropriate product of the counting and Lebesgue measures. Nevertheless, any probability measure with common support is a valid dominating measure and, therefore, admits versions that lead to proportional likelihoods. A particularly interesting example, that goes beyond a purely discrete or continuous random variable, are point-mass mixtures.

Consider the probability measure of a r.v. $Y$ such that $P(Y=a_i)=p_i>0$, for $i=1,\ldots,m$ and $\sum_{i=1}^mp_i=p<1$, and $Y=Z_j$ w.p. $q_j$, such that $Z_j$ is a continuous r.v. on $B_j\subset\mathds{R}$ with (continuous) Lebesgue density $f_j$, for $j=1,\ldots,n$ and $\sum_{j=1}^nq_j=1-p$. In this case, \cite{Raf} show that the probability measure $P$ of $Y$ is dominated by the measure $\nu_1+\nu_2$, where $\nu_1$ is the counting measure and $\nu_2$ is the Lebesgue measure and
\begin{equation}\label{Mx}
\ds \frac{dP}{d(\nu_1+\nu_2)}(y)=\sum_{i=1}^mp_i\mathbb{I}_{a_i}(y)+\sum_{j=1}^nq_jf_j(y)\mathbb{I}_{B_j\setminus A}(y),
\end{equation}
where $A=\{a_i,\ldots,a_n\}$. Typical point-mass mixtures consider the $B_j$'s to be the same. The use of a non-valid RN derivative, in particular by ignoring the indicator functions in (\ref{Mx}), leads to misspecified likelihood functions with possibly serious consequences in the inference process. The density in (\ref{Mx}) is uniquely defined on $A$ and one should always consider continuous versions of the $f_j$'s (in $B_j$) when these exist. These versions not only guarantee the proportionality of likelihoods obtained for different dominating measures (see Theorem \ref{CVT}) as it also guarantees that the likelihood obtained is the limit in (\ref{ExLim}).

The result from \cite{Raf} is actually more general and provides a valid dominating measure with the respective RN derivative for probability measures consisting of a countable mixture of mutually singular probability measures.

\subsection{Exponential families}

A parametric family $\mathcal{P}=\{P_{\theta};\ \theta\in\Theta\}$ dominated by a $\sigma$-finite measure $\nu$ on $(\Omega,\mathcal{F})$ is called an exponential family if and only if
\begin{equation}\label{eq.exp}
\frac{dP_\theta}{d\nu}(\omega)=exp\{[\eta(\theta)]^{\tau}T(\omega)-\xi(\theta)\}h(\omega),\ \omega\in\Omega,
\end{equation}
where $T$ is a random $p$-vector with $p\in{\N}$, $\eta$ is a function from $\Theta$ to $\R^p$, $h$ is a non-negative Borel function and
$\xi(\theta)=log\{\int_{\Omega} exp\{[\eta(\theta)]^{\tau}T(\omega)\}h(\omega)d\nu(\omega)$. For a detailed account about exponential families, please reference to \citet{Jorgensen12}.

Note that the Definition of exponential family above depends on the measure $\nu$. Then, if we change the measure that will dominate the family $\mathcal{P}$, the representation given in \eqref{eq.exp} will be different. Thus, it is natural to ask if the exponential representation is independent of the choice of the dominating measure, i.e., if $\mathcal{P}$ is dominated by a $\sigma$-finite measure $\mu$, then $\frac{dP_\theta}{d\mu}$ has the form given in \eqref{eq.exp}. Before, answering this question, though, we state the following result related to exponential families.

\begin{prop}\label{EXPF1}
For any $A\in\mathcal{F}$, define $\lambda(A)=\int_A h d\nu$, for $h$ as defined in \eqref{eq.exp}. Then, $\lambda$ is a $\sigma$-finite measure on $(\Omega,\mathcal{F})$ and $\mathcal{P}<<\lambda$. Furthermore, $\ds\frac{dP_\theta}{d\lambda}(\omega)=exp\{[\eta(\theta)]^{\tau}T(\omega)-\xi(\theta)\},\ \omega\in\Omega$.
\end{prop}
\begin{proof}
Let $B=\{h>0\}$ and $\lambda(A)=0$ for some $A\in\mathcal{F}$. Then, $\lambda(A)=\lambda(A\cap B)$ and $\int_{A\cap B} h d\nu=0$. Since the function $h$ is strictly positive on $A\cap B$, it follows that $\nu(A\cap B)=0$ and $P_\theta(A)=P_\theta(A\cap B)=0$ for all $\theta\in\Theta$. The expression for $\frac{dP_\theta}{d\lambda}$ follows from the RN chain rule.
\end{proof}

We now move to the main result about exponential families.
\begin{teo}\label{teo.exp}
Being an Exponential family is a property of the model $\mathcal{P}$, i.e., it is independent of the dominating measure $\nu$ used in (\ref{eq.exp}). Moreover, if $\mathcal{P}$ is an Exponential family, then, for all $\sigma$-finite measure $\nu$ such that $\mathcal{P}<<\nu$, there exist functions $\eta$, $T$ and $\xi$ and there exists a measurable function $h_\nu$ such that
$$\frac{dP_\theta}{d\nu}(\omega)=exp\{[\eta(\theta)]^{\tau}T(\omega)-\xi(\theta)\}h_\nu(\omega),\ \omega\in\Omega, \forall\theta\in\Theta.$$
\end{teo}
\begin{proof}
See Appendix.
\end{proof}

\subsection{Missing data problems}\label{subsecMD}

Consider a statistical model $\mathcal{P}=\{P_{\theta};\ \theta\in\Theta\}$ on $(\Omega,\mathcal{F})$, such that $\Omega=\Omega_1\times\Omega_2$ and $\mathcal{F}=\sigma(\mathcal{F}_1\times\mathcal{F}_2)$. Suppose, however, that only $\omega_1\in\Omega_1$ is observed. This is the general formulation of a statistical missing data problem and may be motivated by modelling reasons and/or because the marginal density of $P_{\theta}$ (w.r.t. some dominating measure) on $(\Omega_1,\mathcal{F}_1)$ is not available but the joint density on $(\Omega,\mathcal{F})$ is \citep[see, for example,][]{G&G}. A likelihood-based inference approach considers the (pseudo-)likelihood, which is obtained from the density of $P_{\theta}$ w.r.t. some dominating measure, and integrates out the missing data somehow. This is typically done via EM (or Monte Carlo EM) under the frequentist approach or via MCMC under the Bayesian approach. Both methodologies involve dealing with the conditional measure of the missing data given the data $\omega_1$ and the parameters $\theta$.

Suppose that two dominating measures $\nu_1$ and $\nu_2$ for $\mathcal{P}$ are available. Each of them may then be used to obtain a RN derivative for measures $P_{\theta}$ and, consequently, a (pseudo-)likelihood. Supposing that $\omega_1$ is observed, we have
\begin{equation}
\ds  \pi_{\theta,i}(\omega_2|\omega_1) \propto \pi_{\theta,i}(\omega_1,\omega_2),\;\;i=1,2,
\end{equation}
where the right-hand side is the RN derivative of $P_{\theta}$ w.r.t. $\nu_i$. This way, the left-hand side is the density of the conditional measure of the missing data given the data w.r.t. some dominating measure which is induced by $\nu_i$ and, therefore, may be different for $\nu_1$ and $\nu_2$.

Theorem \ref{main.theo} guarantees that the (pseudo-)likelihood is proportional w.r.t. $\theta$ only - not w.r.t. $\omega_2$, which also needs to be estimated (dealt with). As a consequence, although both measures can be used, this choice may have great influence when devising the inference methodology.
The EM algorithm requires computing an expectation w.r.t. the conditional measure of the missing data whilst the Monte Carlo EM and the MCMC require sampling from this measure. If the conditional densities $\ds \pi_{\theta,i}(\omega_2|\omega_1)$ are different for $i=1$ and $i=2$, it may be the case that the required tasks are harder or even not feasible for one of them - although both densities are valid.

\subsection{Poisson processes}\label{subsecPP}

The Poisson process (PP) is the most common statistical model to fit point pattern data. Consider some region $S\subset\mathds{R}^d$, for $d\in\mathds{N}$ - Poisson processes can actually be defined in more general measurable spaces \citep[see][Chp. 2]{King}. Consider a PP on $S$ with intensity $\lambda:=\{\lambda(s)\geq0,\;\forall s\in S\}$, which defines a probability measure $P_{\lambda}$. In this case, we have two obvious dominating measures for $P_{\lambda}$. The first one represents a realisation $\omega$ as $(N,s_1,\ldots,s_N)$, where $N$ is the number of points and the $s_j$'s are their respective locations. We can factor their joint density as $\pi(N)\pi(s_1,\ldots,s_N|N)$ and use the measure $\nu_3=\sum_{k=0}^{\infty}\nu_{3,k}$, where $\nu_{3,k}=\nu_1\otimes\nu_{2}^k$, as a dominating measure, where $\nu_1$ is the counting measure and $\nu_{2}^k$ is the $k$-dimensional Lebesgue measure. We get that
\begin{equation}\label{LhPP1}
\ds \frac{dP_{\lambda}}{d\nu_3}(\omega)=\frac{1}{N!}\exp\left\{-\int_S \lambda(s)ds\right\}\left(\int_S \lambda(s)ds\right)^N\prod_{j=1}^N\left(\frac{\lambda(s_j)}{\int_S \lambda(s)ds}\right).
\end{equation}

Another valid dominating measure is the probability measure $\nu$ of any PP for which the intensity function is positive everywhere in $S$, in particular constant and equals to 1. In that case, the RN derivative is given by Jacod's formula \citep[see][Corollary II.7.3]{ander93}:
\begin{equation}\label{LhPP2}
\ds \frac{dP_{\lambda}}{d\nu}(\omega)=\exp\left\{-\int_S (\lambda(s)-1) ds\right\}\prod_{j=1}^N(\lambda(s_j)/1).
\end{equation}

Note that the densities in (\ref{LhPP1}) and (\ref{LhPP2}) are proportional in $\lambda$. In a standard inference problem where $\omega$ is observed and $\lambda$ is to be estimated, there is no (practical) difference in considering one or the other. In a more complex context, however, it may be a crucial choice, for example, if the process is not fully observed - see Section \ref{subsecMD}.

If $S\subset\mathds{R}$ and we consider the Skorokhod space $D$ of càdlàg functions with the respective Skorokhod topology, we get that the density in (\ref{LhPP2}) is continuous in $D$ and this is a separable space.

\subsection{Diffusion processes}

Brownian motion driven stochastic differential equations (SDE), known as diffusion processes, are quite popular in the statistical literature to model a variety of continuous time phenomena. Formally, a diffusion is defined as the continuous time stochastic process which is the (unique) solution of a (well-defined) SDE. Making statistical inference for diffusions is a challenging problem due to the complex nature of such processes. The continuous time feature implies that they lie on infinite-dimensional spaces and typically have unknown (intractable) transition densities. As a consequence, an exact likelihood in a discretely observed context is unavailable. The most promising solutions available stand out for treating the inference problem without resorting to discretisation schemes \citep[see][]{bpr06a}. These methodologies, called exact, rely on the (pseudo-)likelihood function of a continuous-time trajectory and give rise to interesting issues related to the context of this paper. We discuss the case where the processes are univariate and the diffusion process $Y:=\{Y_s,\;s\in[0,t]\}$ is defined as the solution for a SDE of the type:
\begin{equation}\label{SDE1}
\ds dY_s=a(Y_s,\theta)ds+\sigma(Y_s,\theta)dW_s,\;s\in[0,t],\;\;Y_0=y_0,
\end{equation}
where $W_s$ is a Brownian motion and functions $a$ and $\sigma$ are supposed to satisfy some regularity conditions to guarantee the existence of an unique solution \citep[see][]{KP95}. Diffusion processes trajectories' are a.s. continuous and non-differentiable everywhere.

In a typical statistical problem, one is interested in estimating the functions $a(Y_s,\theta)$ and $\sigma(Y_s,\theta)$. These are typically defined parametrically, as it is done here, but non-parametric approaches may be considered. In the parametric case, the aim is to estimate the parameter set $\theta$. As it was mentioned above, exact methodologies rely on the likelihood of a complete trajectory which can only be obtained if a valid dominating measure is available. It turns
out, however, that processes with distinct diffusion coefficient $\sigma$ define mutually singular probability measures. As a consequence, there exists no $\sigma$-finite measure that simultaneously dominates the family of probability measures if this is uncountable, which is often the case (if it is countable, a countable sum of measures would dominate - see \cite{Raf}).

Therefore, different values of $\theta$ define mutually singular measures and no likelihood function can be obtained. The solution for this problem considers two transformations of the diffusion path - proposed in \cite{stramer} in a discrete approximation context. A complete path is decomposed as $(Y_{obs},\dot{X})$, where $Y_{obs}$ are the discrete observations of $Y$ and $\dot{X}$ are transformed bridges between the observations. More specifically, for (time-ordered) observations $y_0,\ldots,y_n$ at times $t_0,t_1,\ldots,t_n$, consider the Lamperti transform $X_s=\eta(Y_s,\theta)=\int_{y}^{X_s}\frac{1}{\sigma(u,\theta)}du$, for some element $y$ of the state space of $Y$. This implies that $X$ is the solution of a SDE with unit diffusion coefficient and some drift $\alpha(X_s,\theta)$ (which depends on functions $a$ and  $\sigma$). Now, defining $x_i(\theta)=\eta(y_i,\theta)$, $i=0,\ldots,n$, consider the following transformation of the bridges of $X$ between the $x_i(\theta)$ points, $\dot{X_s}=\varphi^{-1}(X_s)= X_s -\left(1-\frac{s-t_{i-1}}{t_i-t_{i-1}}\right)x_{i-1}(\theta)-\left(\frac{s-t_{i-1}}{t_i-t_{i-1}}\right)x_{i}(\theta)$, for $s\in(t_{i-1},t_i)$. This implies that the transformed bridges start and end in zero and are, therefore, dominated by the measure of standard Brownian bridges. The density of $(Y_{obs},\dot{X})$ is decomposed as $\pi(Y_{obs},\dot{X})=\pi(Y_{obs})\pi(\dot{X}|Y_{obs})$ and obtained w.r.t. to the parameter-free dominating measure $\nu^n\otimes\mathbb{W}^n$ - the product measure of the $n$-dimensional Lebesgue measure and the product measure of standard Brownian bridges of respective time lengths. Lemma 2 from \citet{bpr06a} gives that:
\begin{eqnarray}\label{GT1}
\ds && \pi(Y_{obs},\dot{X})=\prod_{i=1}^n\eta'(y_i;\theta)\phi\left((x_i(\theta)-x_{i-1}(\theta))/\sqrt{t_i-t_{i-1}}\right)  \nonumber\\
&&\exp\left\{\Delta A(x_0(\theta),x_n(\theta);\theta)-\int_{0}^T\left(\frac{\alpha^2+\alpha'}{2}\right)(\varphi_{\theta}(\dot{X}_s);\theta)ds\right\},
\end{eqnarray}
where $\Delta A(x_0(\theta),x_n(\theta);\theta)=A(x_n(\theta);\theta)-A(x_0(\theta);\theta)$, $A(u;\theta)=\int_0^u\alpha(z,\theta)dz$ and $\phi$ is the standard Gaussian density.

Assuming that $\sigma$ is continuously differentiable, one can show that, under the supremum norm, the density in (\ref{GT1}) is continuous in $C$ - the space of continuous functions on $[0,t]$. This ($\sup$ norm on $C$) also defines a separable space.

\section{Final remarks}

In this paper, we discussed some mathematical foundations of Likelihood Theory, more specifically, the definition of likelihood function (in both parametric and non-parametric contexts). We consider the general definition of likelihood function in terms of the Radon-Nikod\'{y}m derivative of each probability measure in the model w.r.t. any dominating measure, evaluated at the observed sample. The Likelihood Proportionality Theorem validates this definition in terms of the Likelihood Principle by guaranteeing the existence of versions of the densities that are a.s. (under every probability measure in the model) proportional for any two dominating measures.

Whilst the Likelihood Proportionality Theorem only guarantees the existence of versions that are proportional, a practical strategy to find such versions is provided by considering continuous versions of densities. Under some mild conditions, continuous versions are shown to always be a.s. proportional and, in many cases, unique \citep[][]{Picc82}. Namely, the use of continuous versions will always be in accordance with the Likelihood Principle. The prior predictive measure is also discussed as a potential choice for dominating measure.

The decision of which dominating measure to use is particularly interesting in cases where there exists no or more than one obvious choice. Both cases are illustrated and discussed in the examples presented in Section \ref{secexamp}. In particular, we present appealing versions of RN derivatives and discuss how different choices, although leading to the same result, may have an influence in the complexity of the inference process.

\section*{Acknowledgements}

The authors would like to thank Gareth Roberts for insightful and stimulating discussions about the paper. The first author would like to thank CNPq-Brasil and FAPEMIG for financial support.

\bibliographystyle{apalike}
\bibliography{biblio}


\section*{Appendix A - Important results and definitions}

We consider the following definitions and results from \cite{Picc82}.

\begin{definition*}
A measure $\nu$ defined on $(\Omega,\mathcal{F})$ is said to be locally finite (LF) if for every point $\omega\in\Omega$ there exists a neighbourhood $U_\omega$ such that $\nu(U_\omega)<\infty$.
\end{definition*}

\begin{thm*}
Any LF measure on $(\Omega,\mathcal{F})$ is $\sigma$-finite.
\end{thm*}

\begin{definition*}
A point $\omega\in\Omega$ is called impossible for the measure $\nu$ on $(\Omega,\mathcal{F})$ if there exists a measurable (open) neighbourhood $U$ of $\omega$ such that $\nu(U)=0$. The set of the points of $\Omega$ which are not impossible for $\nu$ is called its support and it will be denoted by $S_\nu$.
\end{definition*}

\begin{proposition*}
The support of any LF measure on $(\Omega,\mathcal{F})$ is not empty.
\end{proposition*}

\begin{thm*}
The support of a LF measure $\nu$ on $(\Omega,\mathcal{F})$ is a closed set with measure $\nu(\Omega)$.
\end{thm*}

\begin{prop}\label{TCVLF}
Let $\nu$ and $\mu$ be measures on $(\Omega,\mathcal{F})$ and let $S_\nu$ and $S_\mu$ the supports of $\nu$ and $\mu$, respectively. If $\nu<<\mu$, then $S_\nu\subset S_\mu$.
\end{prop}
\begin{proof}
For any $\omega\notin S_\mu$, there exists an open set $U_\omega$ such that $\mu(U_\omega)=0$. Because $\nu<<\mu$, it follows that $\nu(U_\omega)=0$. Then, $\omega \notin S_\nu$ and $S_\mu^c\subset S_\nu^c$.
\end{proof}

The following result from \citet{Picc82} guarantees the uniqueness of continuous versions of densities under some mild conditions.

\begin{thm*}
Let $\mu$ and $\nu$ be LF measures on $(\Omega,\mathcal{F})$ such that $\mu<<\nu$ and $S_{\mu}=S_{\nu}=\Omega$. If there exists a continuous version of $d\mu/d\nu$ on $\Omega$, it is unique.
\end{thm*}

The following variate of the previous theorem is of particular interest in the results presented in this paper.

\begin{teo}\label{uniqcv}
Let $\mu$ and $\nu$ be LF measures on $(\Omega,\mathcal{F})$ such that $\mu<<\nu$. If there exists a continuous version of $d\mu/d\nu$ on $S_{\mu}$, it is unique.
\end{teo}
\begin{proof}
Simply use Proposition \ref{TCVLF}, consider the measures $\mu\big|_{S_{\mu}}$ and $\nu\big|_{S_{\mu}}$ and apply the previous theorem.
\end{proof}

We also consider the following auxiliary result, to be used in the proof of Lemma \ref{theo.2}.

\noindent{\bf Auxiliary result I.} Let $(\Omega,\mathcal{F},\nu)$ be a measure space and $f:\Omega:\longrightarrow\R$ be a real function in $M(\Omega,\mathcal{F})$. Let $A\in\mathcal{F}$ such that $\nu(A)>0$ and $f(\omega)>0$ for all $\omega\in A$. Then $\int_A fd\nu>0$.

\section*{Appendix B - Proofs}

\noindent\textit{Proof of Proposition \ref{propmdm}.} Since $\Upsilon\neq\emptyset$, there exists $\nu\in\Upsilon$ such that $\mathcal{P}<<\nu$. Then, it follows from Lemma 1 that there exists a measure $\lambda$ such that $\mathcal{P}<<\lambda$ and where $\lambda=\sum_{i=1}^{\infty} c_{i}P_{\theta_i}$, where the $c_{i}$'s are nonnegative constants with $\sum_{i=1}^{\infty} c_{i}=1$ and $P_{\theta_i}\in \mathcal{P}$. We now show that measure $\lambda$ is a minimal dominating measure w.r.t. $\mathcal{P}$, i.e. if $\nu\in\Upsilon$, then $\lambda<<\nu$. Take any $\nu\in\Upsilon$ and let $A\in\mathcal{F}$ such that $\nu(A)=0$. Then, $P_\theta(A)=0$ for all $\theta\in\Theta$ and, particularly, $P_{\theta_i}(A)=0$, for all $i\in\N$. Thus, $\lambda(A)=\sum_{i=1}^{\infty} c_{i}P_{\theta_i}(A)=0$.

\vspace{0.5cm}

\noindent\textit{Proof of Theorem \ref{main.theo}.} Let $\nu$ be a minimal dominating measure for $\mathcal{P}$ (its existence is guaranteed by Proposition \ref{propmdm}).
Now, take $h_1\in[\frac{d\nu}{d\nu_1}]_{\nu}$, $h_2\in[\frac{d\nu}{d\nu_2}]_{\nu}$ and, for each $\theta\in\Theta$, take $g_\theta\in[\frac{dP_\theta}{d\nu}]_{\nu}$. Define, for each $\theta\in\Theta$, $f_{1, \theta}(\omega)=g_\theta(\omega)h_1(\omega)$ and $f_{2, \theta}(\omega)=g_\theta(\omega)h_2(\omega)$. It follows that $f_{1,\theta}\in [\frac{dP_{\theta}}{d\nu_1}]_{\nu}$ and $f_{2,\theta}\in  [\frac{dP_{\theta}}{d\nu_2}]_{\nu}$.  Let
$$A=\{\omega\in\Omega;\ h_2(\omega)>0\}$$
so that $\nu(A^c)=0$ and consequently $P_\theta(A)=1$ for all $\theta\in\Theta$. Let $h$ be defined as
$$h(\omega)=\left\{\begin{array}{rc}
\frac{h_1(\omega)}{h_2(\omega)}, &\mbox{if}\quad \omega\in A,\\
0, &\mbox{if}\quad \omega\in A^c.
\end{array}\right.$$
Then, $h\in M(\Omega,\mathcal{F})$ and
$$f_{1,\theta}(\omega)=h(\omega)f_{2,\theta}(\omega),\ \forall\theta\in\Theta,\ \forall\omega\in A.$$

\vspace{0.5cm}

\noindent\textit{Proof of Proposition \ref{T1FT}.}

To prove i), for each $\theta\in\Theta$, take $g_{\theta}^*\in[\frac{dP_{\theta}}{dQ}]_{\nu}$ in $(\Omega,\sigma(T))$ and $h_1\in[\frac{dQ}{d\nu_1}]_{\nu}$ in $(\Omega,\mathcal{F})$. Then, there exists a $\mathcal{B}$-measurable function $g_\theta$ such that $g_{\theta}^*=g_{\theta}\circ T$ \citep[see][Section 1.4, Lemma 1.2]{Shao}. Now, since $T$ is a sufficient statistic for $\mathcal{P}$, it follows from that $g_{\theta}\circ T\in[\frac{dP_{\theta}}{dQ}]_{\nu}$ in $(\Omega, \mathcal{F})$ \citep[see][Section 2.6, Theorem 8]{TSH}. Define the function $f_{1,\theta}$ as
$$f_{1,\theta}(\omega)=g_{\theta}(T(\omega))h_{1}(\omega),\ \forall\omega\in\Omega.$$
Thus, it follows from the RN chain rule, that $f_{1,\theta}\in[\frac{dP_{\theta}}{d\nu_1}]_{\nu}$ for all $\theta\in\Theta$.\\
To prove ii), let $f_{1,\theta}(\omega)=g_{\theta}(T(\omega))h_{1}(\omega)$ and $f_{2,\theta}(\omega)=g_{\theta}(T(\omega))h_{2}(\omega)$ for all $\omega\in\Omega$ and $\theta\in\Theta$, where $h_2\in[\frac{dQ}{d\nu_2}]_{\nu}$. Let $A=\{\omega\in\Omega;\ h_1(\omega)>0\}$. Then, $\nu(A^c)=0$ and $f_{1,\theta}\propto f_{2,\theta}$ in $A$, for all $\theta\in\Theta$.

\vspace{0.5cm}

\noindent\textit{Proof of Lemma \ref{theo.2}.}
Part i. Let $\nu$ be a minimal dominating measure and $Q=\sum_ic_iP_{\theta_i}$ be the measure from Lemma \ref{lema1}. Define $A_i=\{\omega\in\Omega:\frac{dP_{\theta_i}}{d\nu}(\omega)>0\}$ and $A=\bigcup_iA_i$. Thus, $P_{\theta_i}(A_i)=1,\forall i\in\mathds{N}$, $P_{\theta_i}(A)=1,\forall i\in\mathds{N}$ and $Q(A)=1$. Notice that $Q(A^c)=0\Rightarrow P_{\theta}(A^c)=0,\forall \theta\in\Theta$ and, therefore, $P_{\theta}(A)=1,\forall\theta\in\Theta$.

Part ii. Let $B\subset A$ such that $\nu_2(B)=0$. Assume that $\nu_1(B)>0$. Since $B=\bigcup_i(A_i\bigcap B)$, there must exist $i_0\in\mathds{N}$ such that $\nu_1(A_{i_0}\bigcap B)>0$. By the auxiliary result I, as $\frac{dP_{\theta_{i_0}}}{d\nu}(\omega)>0$, for all $\omega\in A_{i_0}$, the latter inequality implies that
$$\ds \int_{A_{i_0}\bigcap B}\frac{dP_{\theta_{i_0}}}{d\nu}d\nu_1>0.$$
Now define $C:=\{\omega \in \Omega:\frac{d\nu}{d\nu_1}(\omega)>0\}$ and note that $\nu_1(A_{i_0}\bigcap B \bigcap C)>0$. Then, by the auxiliary result I,
$$\ds P_{\theta_{i_0}}(A_{i_0}\bigcap B)=\int_{A_{i_0}\bigcap B}\frac{dP_{\theta_{i_0}}}{d\nu}\frac{d\nu}{d\nu_1}d\nu_1=
\int_{A_{i_0}\bigcap B \bigcap C}\frac{dP_{\theta_{i_0}}}{d\nu}\frac{d\nu}{d\nu_1}d\nu_1>0,$$
which is a contradiction, since by $\nu_2(B)=0$ we should conclude that $P_{\theta_{i_0}}(A_{i_0}\bigcap B)=0$. Therefore, $\nu_1(B)$ must be zero.

\vspace{0.5cm}

\noindent\textit{Proof of Theorem \ref{CVT}.}

Let $\{P_{\theta_i}\}$ be a family of probability measures used in the construction of the measure $\nu$. Now, define measures $\dot{P}_{\theta}$, $\dot{\nu_1}$, $\dot{\nu_2}$ and $\dot{\nu}$ to be the restriction of the respective measures on $(A,\mathcal{F}(A))$, for all $\theta\in\Theta$.
For each $i\in\N$, consider the continuous derivatives $f_{1,\theta_i}\in [\frac{d\dot{P}_{\theta_i}}{d\dot{\nu}_1}]_{\dot{\nu}}$, $f_{2,\theta_i}\in [\frac{d\dot{P}_{\theta_i}}{d\dot{\nu}_2}]_{\dot{\nu}}$ and take any $h\in[\frac{d\dot{\nu}_2}{d\dot{\nu}_1}]_{\dot{\nu}}$.
For each $i\in\N$, define $A_i=\{\omega\in A;\ f_{1,\theta_i}(\omega)=h(\omega)f_{2,\theta_i}(\omega)\}$ and note that the RN chain rule implies that $\nu(A_i^c)=0$ for all $i\in\N$.
Now, let $B_i=\{\omega\in A;\ f_{2,\theta_i}(\omega)>0\}$, $B=\bigcup_{i=1}^\infty B_i$, $D_h=\bigcap_{i=1}^\infty A_i$ and $S_h=D_h\cap B$. It follows that $\nu(B)=1$, $\nu(S_h)=1$ and, consequently, $P_{\theta}(S_h)=1$, for all $\theta\in\Theta$. Furthermore, $h$ is continuous in the subspace $S_h$. To see that, let $\omega_0\in S_h$ and $\{\omega_n\}_{n=1}^{\infty}\subset S_h$ such that $\lim_n \omega_n=\omega_0$. It follows that $\omega_0\in D_h$ and there exists $i_0\in\N$ such that $\omega_0\in B_{i_0}$. This implies that
\begin{equation}\label{gcont1}
h(\omega_0)=\frac{f_{1,\theta_{i_0}}(\omega_0)}{f_{2,\theta_{i_0}}(\omega_0)}.
\end{equation}

\noindent Furthermore, since the function $f_{2,\theta_{i_0}}$ is continuous in $A$, $S_h\cap B_{i_0}$ is an open set in $S_h$. Thus, by the convergence of the sequence $\{\omega_n\}_{n=1}^{\infty}$, there exists $n_0\in\N$ such that, for $n\ge n_0$, $\omega_n\in S_h\cap B_{i_0}$ and
\begin{equation}\label{gcont2}
h(\omega_n)=\frac{f_{1,\theta_{i_0}}(\omega_n)}{f_{2,\theta{i_0}}(\omega_n)}.
\end{equation}

\noindent Finally, from (\ref{gcont1}) and (\ref{gcont2}) and the continuity of $f_{1,\theta_{i_0}}$ and $f_{2,\theta_{i_0}}$, it follows that
$$\lim_n h(\omega_n)=h(\omega_0),$$
\noindent which establishes the continuity of $h$ in $S_h$.

\noindent Now, for each $\theta\in\Theta$, define the following set
$$B_\theta=\{\omega\in S_h;\ f_{1,\theta}(\omega)=h(\omega)f_{2,\theta}(\omega)\}.$$
It follows, by the RN chain rule, that $\nu(B_\theta^c\bigcap S_h)=0$, for all $\theta\in\Theta$. Since the function $(f_{1,\theta}-hf_{2,\theta})$ is continuous on $S_h$, we have that $B_\theta$ is a closed set in $S_h$ for each $\theta\in\Theta$ and, consequently, $B_h=\bigcap_{\theta\in\Theta} B_\theta$ is also a closed set in $S_h$. Since $S_h$ is a subspace of a metric separable space, $S_h$ is also a metric separable space. This implies that there exists a sequence $\{\theta_j\}\subset\Theta$ such that $B_h=\bigcap_{j=1}^\infty B_{\theta_j}$. Moreover, since $\nu(B_\theta^c\bigcap S_h)=0,\;\forall\theta\in\Theta$, it follows that $\nu_1(B_h^c\bigcap S_h)=0$ which, in turn, implies that $P_\theta(B_h)=1$ for each $\theta\in\Theta$, and
$$f_{1,\theta}(\omega)=h(\omega)f_{2,\theta}(\omega),\ \forall\theta\in\Theta,\ \forall\omega\in B_h.$$

\vspace{0.5cm}

\noindent\textit{Proof of Corollary \ref{CVTcor1}.}
Since $f_{1,\theta}$ and $f_{2,\theta}$ are strictly positive in $\Omega$, for all $\theta\in\Theta$, it follows that all the $P_{\theta}$'s, $\nu_1$ and $\nu_2$ are equivalent and, by Proposition \ref{TCVLF}, $S_{\theta}=S_{\nu_2}=S_{\nu_1}=\Omega$, for all $\theta\in\Theta$. For each $\theta\in\Theta$, define $h_{\theta}(\omega)=\frac{f_{1,\theta}(\omega)}{f_{2,\theta}(\omega)}$, for all $\omega\in\Omega$, and note that, for all $\theta\in\Theta$, $h_{\theta}\in[\frac{d\nu_2}{d\nu_1}]_{\nu}$ and $h_{\theta}$ is continuous in $\Omega$. Since, $\nu_1$ and $\nu_2$ are LF measures, Theorem \ref{uniqcv} guarantees that all the $h_{\theta}$'s coincide in $\Omega$, i.e. $h_{\theta}=h$, for all $\theta\in\Theta$. The result follows from the fact that
$f_{1,\theta}(\omega)=h(\omega)f_{2,\theta}(\omega)$, for all $\omega\in\Omega$ and for all $\theta\in\Theta$.

\vspace{0.5cm}

\noindent\textit{Proof of Proposition \ref{PMDM1}.}
Let $\mu$ be a $\sigma$-finite measure such that $\mathcal{P}<<\mu$. Let $g_\theta(x)=\frac{dP_\theta}{d\mu}(x)$ and define
$$m^*(x)=\int_\Theta g_\theta(x)dR.$$
\noindent Now consider the predictive measure $\xi$ obtained from $m^*$, i.e.,
$$\xi(A)=\int_A m^* d\mu,\ \forall A\in\mathcal{B}_\mathcal{X}.$$
\noindent We claim that $\lambda=\xi$. For any $A\in\mathcal{B}_\mathcal{X}$,
\begin{eqnarray*}
\lambda(A)&=&\int_A m d\nu=\int_A\int_\Theta f_\theta(x)dR d\nu
\overset{(i)}{=}\int_\Theta\int_A f_\theta(x)d\nu dR=\int_\Theta P_\theta(A)dR\\
&=&\int_\Theta\int_A g_\theta(x)d\mu dR\overset{(ii)}{=}\int_A\int_\Theta g_\theta(x)dR d\mu
=\int_A m^* d\mu=\xi(A),
\end{eqnarray*}
\noindent where the equalities $(i)$ and $(ii)$ follow from Fubini's theorem.

\vspace{0.5cm}

\noindent\textit{Proof of Theorem \ref{PMDM2}.}
If $\lambda$ dominates $P_\theta$, the result follows immediately since $\lambda(N)=0$. Suppose now that $P_\theta(N)=0$ and take $A\in\mathcal{B}_\mathcal{X}$ such that $\lambda(A)=0$. We have to show that $P_\theta(A)=0$. Note that
\begin{equation}\label{eq.m}
0=\lambda(A)=\lambda(A\cap N^c)=\int_{A\cap N^c} md\nu.
\end{equation}
\noindent Then, since $m$ is strictly positive in $A\cap N^c$, equation \eqref{eq.m} is true only if $\nu(A\cap N^c)=0$. Hence, $P_\theta(A\cap N^c)=0$. But, by hypothesis, $P_\theta(A)=P_\theta(A\cap N^c)$ and the result follows.

\vspace{0.5cm}

\noindent\textit{Proof of Theorem \ref{PMDM3}.}
\noindent Let $M=M_\theta$ for all $\theta\in\Theta$ and let $A\in\mathcal{B}_\mathcal{X}$ such that $\lambda(A)=0$. To show that $P_\theta(A)$ for all $\theta\in\Theta$ is sufficient to show that $\nu(A\cap M)=0$, since $\mathcal{P}<<\nu$ and $P_\theta(A)=P_\theta(A\cap M)$ for all $\theta\in\Theta$. Suppose that $\nu(A\cap M)>0$. Hence, since $f_\theta$ is strictly positive on $A\cap M$,
\begin{equation}\label{eq.suporte}
P_\theta(A)=P_\theta(A\cap S)=\int_{A\cap S} f_\theta d\nu>0,\ \forall\theta\in\Theta.
\end{equation}
\noindent On the other hand,
\begin{equation}\label{eq.lambda}
\lambda(A)=\int_A m d\nu=\int_A\int_\Theta f_\theta dR d\nu=\int_\Theta P_\theta(A) dR=\int_\Theta P_\theta(A\cap S) dR,
\end{equation}
\noindent where the penultimate equation follows from Fubini's theorem. Then, since $R(\Theta)>0$, it follows from $\eqref{eq.suporte}$ and $\eqref{eq.lambda}$ that $\lambda(A)>0$, contradicting the assumption that $\lambda(A)=0$. So, $\nu(A\cap M)=0$ and the proof is complete.

\vspace{0.5cm}

\noindent\textit{Proof of Theorem \ref{teo.exp}.}
Suppose that $\frac{dP_\theta}{d\nu}$ is given by \eqref{eq.exp}. Consider the measure $Q$ given by Lemma \ref{lema1} and let $q\in [\frac{dQ}{d\nu}]$. Remember that $Q$ is minimal and so $Q<<\nu$. Without loss of generality we may assume that $q>0$. Define, for each $\theta\in\Theta$, the following function:
\begin{equation}\label{function.b}
b_{\theta}(\omega)=exp\{[\eta(\theta)]^{\tau}T(\omega)-\xi(\theta)\}m(\omega),\ \omega\in\Omega,
\end{equation}
\noindent where $m=h_\nu/q$. On the other hand, by RN chain rule, it follows that
\begin{equation}\label{function.rn.q}
   exp\{[\eta(\theta)]^{\tau}T(\omega)-\xi(\theta)\}h_{\nu}(\omega)=\frac{dP_\theta}{dQ}(\omega)q(\omega),\ \nu-a.e.
\end{equation}
\noindent Consequently, from \eqref{function.b} and \eqref{function.rn.q}, $b_{\theta}=dP_{\theta}/dQ$ $\nu$-almost-everywhere. Hence, $b_\theta\in [\frac{dP_\theta}{dQ}]$.
Now, let $\mu$ be a $\sigma$-finite measure such that $\mathcal{P}<<\mu$ and let $\mu\neq\nu$. Again, by the minimality of $Q$, $Q<<\mu$. Let $s\in [\frac{dQ}{d\mu}]$ and define, for each $\theta\in\Theta$,
\begin{equation}\label{eq.exp.2}
p_\theta(\omega)=exp\{[\eta(\theta)]^{\tau}T(\omega)-\xi(\theta)\}h_{\mu}(\omega),\ \omega\in\Omega,
\end{equation}
\noindent where $h_{\mu}=ms$. Hence, by RN chain rule, $p_{\theta}\in [\frac{dP_\theta}{d\mu}]$ and the proof is complete.

\end{document}